\documentclass[12pt]{amsart}

\newtheorem{theorem}{Theorem}[section]
\newtheorem{lemma}[theorem]{Lemma}
\newtheorem{corollary}[theorem]{Corollary}
\newtheorem{conjecture}[theorem]{Conjecture}
\newtheorem{remark}[theorem]{Remark}

 \theoremstyle{definition}
\newtheorem{definition}[theorem]{Definition}

\numberwithin{equation}{section}

\def\d{\delta}

\def\xb{{\mathbf{x}}}

\def\eb{{\mathbf{e}}}
\def\yb{{\mathbf{y}}}

\def\br{\mathbb{R}}

\begin{document}

\title[Genetic Volterra algebras]{genetic Volterra algebras and their derivations}

\author{Rasul Ganikhodzhaev}
\address{Rasul Ganikhodzhaev, Faculty of Mechanics and Mathematics, National University of Uzbekistan,
Vuzgorodok, 100174, Tashkent, Uzbekistan.}

\author{Farrukh Mukhamedov}
\address{Farrukh Mukhamedov, Department of Mathematical Sciences, College of Science, The United Arab Emirates University, P.O. Box 15551, Al Ain, Abu Dhabi, UAE}
\email{far75m@gmail.com, farrukh.m@uaeu.ac.ae}

\author{Abror Pirnapasov}
\address{Abror Pirnapasov, Department of Mathematics, The Abdus Salam International Centre for Theoretical Sciences, Trieste, Italy.}
\email{pirnapasov\_imc@mail.ru}

\author{Izzat Qaralleh}
\address{Izzat Qaralleh, Department of Mathematics, Faculty of
Science, Tafila Technical University, Tafila, Jordan}
\email{izzat\_math@yahoo.com}




\begin{abstract}
The present paper is devoted to genetic Volterra algebras. We first
study characters of such algebras. We fully describe associative
genetic Volterra algebras, in this case all derivations are trivial.
In general setting, i.e. when the algebra is not associative, we
provide a sufficient condition to get trivial derivation on generic
Volterra algebras. Furthermore, we describe all derivations of three
dimensional generic Volterra algebras, which allowed us to prove
that any local derivation is a derivation of the algebra.
\end{abstract}

\maketitle

\section{Introduction}

There exist several classes of non-associative algebras (baric,
evolution, Bernstein, train, stochastic, etc.), whose investigation
has provided a number of significant contributions to theoretical
population genetics \cite{R,WB}. Such classes have been defined
different times by several authors, and all algebras belonging to
these classes are generally called \textit{genetic}. In \cite{Eth}
it was introduced the formal language of abstract algebra to the
study of the genetics. Note that problems of population genetics can
be traced back to Bernstein's work \cite{B0,B} where evolution
operators were studied. Such kind of operators are mostly described
by quadratic stochastic operators (see also \cite{Ly2}).

A quadratic stochastic operator is used to present the time
evolution of species in biology \cite{L,V1}, which arises as
follows. Consider a population consisting of $m$ species (or traits)
which we denote by $I=\{1,2,\cdots,m\}$. Let
$x^{(0)}=\left(x_1^{(0)},\cdots,x_m^{(0)}\right)$ be a probability
distribution of species at an initial state and $p_{ij,k}$ be a
probability that individuals in the $i^{th}$ and $j^{th}$ species
(traits) interbreed to produce an individual from $k^{th}$ species
(trait). Then a probability distribution
$x^{(1)}=\left(x_{1}^{(1)},\cdots,x_{m}^{(1)}\right)$ of the spices
(traits) in the first generation can be found as a total
probability,  i.e.,
\begin{equation*} x_k^{(1)}=\sum_{i,j=1}^m p_{ij,k} x_i^{(0)} x_j^{(0)}, \quad k\in\{1,\dots,m\}.
\end{equation*}
This means that the correspondence $x^{(0)} \to x^{(1)}$ defines a
mapping $V$ called \textit{the evolution operator}. This mapping,
for a free population, is a quadratic mapping of the simplex of all
probability distribution on $I$. Therefore, such an operator is also
called \textit{quadratic stochastic operator} (QSO). In other words,
a QSO describes a distribution of the next generation if the
distribution of the current generation was given. The fascinating
applications of QSO to population genetics were given in
\cite{HHJ,Ly2}.
%
In \cite{11,MG2015}, it was given along
self-contained exposition of the recent achievements and open
problems in the theory of the QSO.

Note that each QSO defines an algebraic structure on the vector
space $\br^m$ containing the simplex (see next section for
definitions). Such an algebra is called \textit{genetic algebra}.
Several works are devoted (see\cite{Ly2,R}) to certain properties of
these algebras. We point out that the algebras that arise in
genetics (via gametic, zygotic, or copular algebras) have very
interesting structures. They are generally commutative but
nonassociative, yet they are not necessarily Lie, Jordan, or
alternative algebras. In addition, many of the algebraic properties
of these structures have genetic significance. For example, a more
modern use of the genetic algebra theory to self fertilization can
be found in \cite{H1}. Therefore, it is the interplay between the
purely mathematical structure and the corresponding genetic
properties that makes this subject so fascinating. We refer to
\cite{WB} for the comprehensive reference.

In population genetics, it is important to study dynamics of
so-called Volterra operators \cite{NSE}. The dynamics of such kind
of operators have been investigated in \cite{G}. However, genetic
algebras associated to these operators were not completely studied
yet. Therefore, in the present paper, we are going systematically
investigate these kind of algebras (see Sections 3 and 4). On the
other hand, there has recently been much work on the subject of
derivations of genetic algebras (for example \cite{Go,MC,MQ2014}).
Certain interpretations of the derivations have been discussed in
\cite{H2}. Moreover, in these investigations, derivations of genetic
Volterra algebras were not studied. In this paper, we fully describe
associative genetic Volterra algebras, in the later case, all
derivations are trivial. In section 5, we consider a general
setting, i.e. the algebra is not necessarily associative. In this
case, we provide a sufficient condition to get a trivial derivation
on generic Volterra algebra. Furthermore, we describe all
derivations of three dimensional generic Volterra algebra, which
allowed us to prove that any local derivation is a derivation of the
algebra.

\section{Preliminaries}

This section is devoted to some necessary notations.

Let $ I = \{ 1, \dots, m \} $. By $\{\eb_i\}_{i\in I}$ we denote the
standard basis in $\br^m$, i.e. $\eb_i=(\d_{i1},\dots,\d_{im})$,
where $\d$ is the Kronecker's Delta.  Throughout this paper, we
consider the simplex:
\begin{eqnarray}\label{1.2}
S^{m-1} = \left\{\textbf{x}=(x_i)\in\br^m\ : \ x_i\geq0, \ \forall
i\in I, \quad \sum\limits_{i=1}^{m}x_i = 1\right\} \label{eqn1.1}.
\end{eqnarray}

A \textit{quadratic stochastic operator (QSO)} is a mapping of the
simplex $S^{m-1}$ into itself of the form
\begin{equation}\label{1.3}
V:x_k^{\prime}=\sum_{i,j=1}^mp_{ij,k}x_ix_j, \ \ k=1,2,\ldots,m
\end{equation}
where $P_{ij,k}$ are heredity coefficients, which satisfy the
following conditions:
\begin{equation}\label{1.4}
p_{ij,k}\geq0, \ \ p_{ij,k}=p_{ji,k}, \ \ \sum_{k=1}^{m}p_{ij,k}=1,
\ \ i,j,k\in\{1,2,\ldots,m\}
\end{equation}

A QSO $V$ defined by \eqref{1.3} is called \textit{Volterra
operator} \cite{G} if one has
\begin{equation}\label{1.5}
p_{ij,k}=0 \ \ \mbox{if} \ \ k\not\in \{i,j\}, \ \ \textrm{for all}\
i,j,k\in I.
\end{equation}

From \eqref{1.4} and \eqref{1.5} we infer that
\begin{equation}\label{1.51}
p_{ii,i}=1 \ \ \mbox{and} \ \ p_{ij,i}+p_{ij,j}=1, \ \ \textrm{for
all} \ i,j\in I, \ (i\neq j).
\end{equation}

\begin{remark} Note that it is obvious that the biological behavior of
condition \eqref{1.5} is that the offspring repeats one of its
parents' genotype (see \cite{G,11}).
\end{remark}

Let $V$ be a QSO and suppose that $\mathbf{x},\mathbf{y}\in
\mathbb{R}^m$ are arbitrary vectors, we introduce a multiplication
rule (see \cite{H}) on $\mathbb{R}^m $ by
%
\begin{equation}\label{4.1}
(\mathbf{x}\circ \mathbf{y})_k=\sum^m_{i,j=1}p_{ij,k}x_iy_j
\end{equation}
where $\mathbf{x}=(x_1,\ldots ,x_m),\mathbf{y}=(y_1,\ldots
,y_m)\in \mathbb{R}^m$.

The pair  $(\br^m,\circ)$ is called \textit{genetic algebra}. We
note the this algebra is commutative, i.e.
$\mathbf{x}\circ\mathbf{y}=\mathbf{y}\circ\mathbf{x}$. Certain
algebraic properties of such kind of algebras were investigated in
\cite{H,Ly2,WB}. In general, the genetic algebra is not necessarily
to be associative. In \cite{G2008,NGS} associativity of low
dimensional genetic algebras have been studied. If $V$ is a Volterra
QSO, then the associated genetic algebra is called {\it genetic
Volterra algebra}.

\begin{remark}
Let $A$ be a Volterra algebra generated by heredity coefficients
$\{p_{ij,k}\}$. Then from \eqref{1.51} and \eqref{4.1} we
immediately find
\begin{itemize}
\item[(a)] for every $i,j\in I$ ($i\neq j$) one has
\begin{equation}\label{hh1}
\eb_{i}\circ\eb_{j}=p_{ij,i}\eb_{i}+p_{ij,j}\eb_{j}, \end{equation}

\item[(b)] $\eb_i^2=\eb_i$ for every $i\in I$.
\end{itemize}
\end{remark}

\begin{theorem}\cite{Ly2} Let $A$ be an algebra over $\br$. If it has a genetic realization with respect to
the natural basis ${\eb_{i},... , \eb_{m}}$, then $A$ is a (non
associative) Banach algebra with respect to the norm
$\|\xb\|=\sum\limits_{i=1}^{m}|x_{i}|$  for
$\xb=\sum\limits_{i=1}^{m}x_{i}\eb_{i}\in A$.
\end{theorem}

Recall that a \textit{derivation} on algebra $(A,\circ)$ is a linear
mapping $D : A \to A$ such that $D(u\circ v)=D(u)\circ v+u\circ
D(v)$ for all $u,v\in A$. It is clear that $D\equiv 0$ is also a
derivation, and such derivation is called \textit{trivial} one.

\section{Characters of genetic Volterra algebras}

In this section, we characterize all characters of genetic Volterra
algebras.

Let $A$ be a genetic Volterra algebra. We recall that a {\it
character} of $A$ is a linear functional $h$ from $A$ to $\br$ with
$h(\xb\circ \yb)=h(\xb)h(\yb)$ for all $\xb,\yb\in A$.

\begin{lemma}\label{har} Let $A$ be an $m$-dimensional genetic Volterra algebra. If $h(\xb)=\sum\limits_{i=1}^{m}a_{i}x_{i}$ is a character of $A$,
then $a_{i}\in\{0,1\}$.
\end{lemma}

\begin{proof}
Let, as before, $\{\eb_i\}$ be vertices of the simplex $S^{m-1}$. It
is clear that $h(\eb_{i})=a_{i}$. Moreover, one can see that each
vector $\eb_i$ is an idempotent of the algebra, i.e.
$\eb^{2}_{i}=\eb_{i}$. This implies that
$$a_{i}=h(\eb_{i})=h(\eb_{i}\circ \eb_{i})=h^{2}(\eb_{i})=a^{2}_{i},$$
which means $a_{i}(a_{i}-1)=0$. Hence, $a_{i}\in\{0,1\}$.
\end{proof}

The proved Lemma \ref{har} implies that for every character $h$ of a
genetic Volterra algebra, one can find a subset $E\subset
I(=\{1,2,\dots,m\})$ such that $h=h_E$, where
\begin{equation}\label{hE}
h_E(\xb)=\sum_{i\in E}x_{i}.
\end{equation}

\begin{theorem}\label{hh0} Let $A$ be a genetic Volterra algebra and $E\subset I$.
Then the following conditions are equivalent:
\begin{itemize}
\item[(i)] The functional $h_E$ is a character;
\item[(ii)] For all  $i\in I\setminus E$, $j\in E$ one has
$p_{ij,j}=0$.
\end{itemize}
\end{theorem}

\begin{proof}
(i)$\Rightarrow $(ii). Assume that $h_E$ is a character. Due to
\eqref{hh1} one finds
\begin{eqnarray}\label{hh2}
h_E(\eb_{i}\circ \eb_{j})=h_E(p_{ij,i}\eb_{i}+p_{ij,j}\eb_{j})
=p_{ij,i}h_E(\eb_{i})+p_{ij,j}h_E(\eb_{j}).
\end{eqnarray}

Then for $i\in I\setminus E$, $j\in E$ from \eqref{hE} one finds
$h_E(\eb_{i})=0$ and $h_E(\eb_{j})=1$. Hence, one gets
\begin{eqnarray*}
0=h_E(\eb_{i})h_E(\eb_{j})=h_E(\eb_{i}\circ \eb_{j})=p_{ij,j}
\end{eqnarray*}
which is the required assertion.

(ii)$\Rightarrow$(i). First note that $h$ is character if and only
if one has $h(\eb_{i})h(\eb_{j})=h(\eb_{i}\circ \eb_{i})$ for all
$i,j\in I$.

To prove the required assertion we consider several cases.

{\sc Case 1.} Assume that $i,j\in E$. Then from \eqref{hE} we have
$h_E(\eb_{i})=h_E(\eb_{j})=1$, and hence from \eqref{hh2} one gets
\begin{eqnarray*}
h_E(\eb_{i}\circ
\eb_{j})=p_{ij,i}+p_{ij,j}=1=h_E(\eb_{i})h_E(\eb_{j}).
\end{eqnarray*}

{\sc Case 2.} Assume that $i,j\in I\setminus E$. Then
$h_E(\eb_{i})=h_E(\eb_{j})=0$. From \eqref{hh2} we have
\begin{eqnarray*}
h_E(\eb_{i}\circ \eb_{j})=0=h_E(\eb_{i})h_E(\eb_{j}).
\end{eqnarray*}

{\sc Case 3.} Let $i\in I\setminus E$, $j\in E$. Then due to our
assumption one has $p_{ij,j}=0$. On the other hand, we find
$h_E(\eb_{i})=0$ and $h_E(\eb_{j})=1$. Again from \eqref{hh2} one
gets
\begin{eqnarray*}
h_E(\eb_{i}\circ \eb_{j})=p_{ij,j}=0=h_E(\eb_{i})h_E(\eb_{j}).
\end{eqnarray*}
This completes the proof.
\end{proof}

\section{Associativity of genetic Volterra algebras}

In this section, we find necessary and sufficient conditions for the
associativity of generic Volterra algebras in terms of the heredity
coefficients.

Before formulation our main result, we prove an auxiliary fact.

\begin{lemma}\label{pp0} Let $\{p_{ij,k}\}$ be the heredity coefficient of a Volterra algebra. Then for any $i,j,k$ with $(i-j)(j-k)(k-i)\neq 0$
the equality
\begin{equation}\label{pp1}
p_{jk,j}p_{ij,i}+p_{jk,k}p_{ik,i}=p_{ij,i}p_{ik,i}
\end{equation}
 is equivalent to
\begin{equation}\label{pp2}p_{ij,i}p_{ik,k}+p_{ij,j}p_{jk,k}=p_{jk,k}p_{ik,k}.\end{equation}
\end{lemma}

\begin{proof} For any $i,j,k$ with $(i-j)(j-k)(k-i)\neq 0$ due to Volterra condition, we have
$$p_{jk,j}=1-p_{jk,k},\ \  p_{ij,i}=1-p_{ij,j}, \ \
p_{ik,i}=1-p_{ik,k}.$$ Therefore, from \eqref{pp1} we have
\begin{eqnarray*}
0&=&p_{jk,j}p_{ij,i}+p_{jk,k}p_{ik,i}-p_{ij,i}p_{ik,i}\\[2mm]
&=&(1-p_{jk,k})(1-p_{ij,j})+p_{jk,k}(1-p_{ik,k})-p_{ij,i}(1-p_{ik,k})\\[2mm]
&=&1-p_{jk,k}-p_{ij,j}+p_{jk,k}p_{ij,j}+p_{jk,k}-p_{jk,k}p_{ik,k}-p_{ij,i}+p_{ij,i}p_{ik,k}\\[2mm]
&=&p_{jk,k}p_{ij,j}+p_{ij,i}p_{ik,k}-p_{jk,k}p_{ik,k}
\end{eqnarray*}
which is \eqref{pp2}.  The proof is complete.
\end{proof}

Now we are ready to formulate a main result of this section.

\begin{theorem}\label{aa1}
Let $A$ be a genetic Volterra algebra. Then the following conditions
are equivalent:
\begin{itemize}
\item[(i)]  $A$ is associative;
\item[(ii)] one has
\begin{itemize}
\item[(a)]  $p_{ij,i}\in \{0,1\}$ for any $i,j\in I$;

\item[(b)] $p_{jk,j}p_{ij,i}+p_{jk,k}p_{ik,i}=p_{ij,i}p_{ik,i}$ for all $i,j,k\in I$
with $(i-j)(j-k)(k-i)\neq 0$.
\end{itemize}
\end{itemize}
\end{theorem}

\begin{proof}  (ii)$\Rightarrow $(i). Assume that (a) and (b) conditions are
satisfied. To prove the associativity, it is enough to establish the
associativity on basis elements $\eb_{1}$, $\eb_{2},\dots ,\eb_{m}$,
i.e.
\begin{equation}\label{aa2}
\eb_{i} \circ (\eb_{j} \circ \eb_{k})=(\eb_{i} \circ \eb_{j}) \circ
\eb_{k}, \ \ \textrm{for all} \ i,j,k\in I.
\end{equation}

Due to the commutativity of the algebra, to show the last equality,
it is sufficient to prove \eqref{aa2} for the cases: $k=j$ and
$(i-j)(j-k)(k-i)\neq 0$, respectively.

The case $k=j$ immediately follows from (a). Therefore, we assume
$(i-j)(j-k)(k-i)\neq 0$. Then using \eqref{hh1} with (b) one finds
\begin{eqnarray*}\eb_{i} \circ (\eb_{j} \circ
\eb_{k})&=&\eb_{i}
\circ(p_{jk,j}\eb_{j}+p_{jk,k}\eb_{k})\nonumber\\[2mm]
&=&p_{jk,j}(p_{ij,i}\eb_{i}+p_{ij,j}\eb_{j})+p_{jk,k}(p_{ik,i}\eb_{i}+p_{ik,k}\eb_{k})\\[2mm]
&=&(p_{jk,j}p_{ij,i}+p_{jk,k}p_{ik,i})\eb_{i}+p_{jk,j}p_{ij,j}\eb_{j}+p_{jk,k}p_{ik,k}\eb_{k}\\[2mm]
&=&p_{ij,i}p_{ik,i}\eb_{i}+p_{jk,j}p_{ij,j}\eb_{j}+(p_{ij,i}p_{ik,k}+p_{ij,j}p_{jk,k})\eb_{k}\\[2mm]
&=&p_{ij,i}(p_{ik,i}\eb_{i}+p_{ik,k}\eb_{k})+p_{ij,j}(p_{jk,j}\eb_{j}+p_{jk,k}\eb_{k})\\[2mm]
&=&(\eb_{i} \circ \eb_{j}) \circ \eb_{k}.
\end{eqnarray*}
This means that $A$ is associative.

(i)$\Rightarrow $(ii). Now we suppose that $A$ is associative, i.e.
\eqref{aa2} holds.

First, we assume that $k=j$. Then due to $\eb^{2}_{j}=\eb_{j}$ we
find
\begin{equation}\label{aa3}
\eb_{i} \circ \eb_{j}=(\eb_{i} \circ \eb_{j}) \circ \eb_{j}.
\end{equation}

Due to \eqref{hh1} one gets
\begin{eqnarray}\label{aa5}
(\eb_{i} \circ \eb_{j}) \circ
\eb_{j}&=&(p_{ij,i}\eb_{i}+p_{ij,j}\eb_{j})\circ \eb_{j}\nonumber\\
&=& p_{ij,i}\eb_{i}\circ \eb_{j}+p_{ij,j}\eb_{j}\circ
\eb_{j}\nonumber\\
&=&p_{ij,i}(p_{ij,i}\eb_{i}+p_{ij,j}\eb_{j})+p_{ij,j}\eb_{j}\nonumber\\
&=&p^2_{ij,i}\eb_{i}+(p_{ij,i}p_{ij,j}+p_{ij,j})\eb_{j}
\end{eqnarray}

Now substituting \eqref{hh1}, \eqref{aa5} into \eqref{aa3}, and
equalizing appropriate coefficients on basis elements, we obtain
\begin{equation*}
p_{ij,i}(1-p_{ij,i})=0
\end{equation*}
which implies $p_{ij,i}\in \{0,1\}$.

Let us assume that $(i-j)(j-k)(k-i)\neq 0$ and consider the equality
\begin{equation}\label{aa6}
\eb_{i} \circ (\eb_{j} \circ \eb_{k})=(\eb_{i} \circ \eb_{j}) \circ
\eb_{k}.
\end{equation}

Keeping in mind \eqref{hh1} from the left side of \eqref{aa6} we get
\begin{eqnarray}\label{aa7}
\eb_{i} \circ (\eb_{j} \circ \eb_{k})&=&\eb_{i}
\circ(p_{jk,j}\eb_{j}+p_{jk,k}\eb_{k})\nonumber\\
&=&p_{jk,j}(p_{ij,i}\eb_{i}+p_{ij,j}\eb_{j})+p_{jk,k}(p_{ik,i}\eb_{i}+p_{ik,k}\eb_{k}).
\end{eqnarray}
Similarly, the right hand side of \eqref{aa6} reduces to
\begin{eqnarray}\label{aa8}
(\eb_{i} \circ \eb_{j}) \circ
\eb_{k}=p_{ij,i}(p_{ik,i}\eb_{i}+p_{ik,k}\eb_{k})+p_{ij,j}(p_{jk,j}\eb_{j}+p_{jk,k}\eb_{k}).
\end{eqnarray}
Now substituting \eqref{aa7},\eqref{aa8} into \eqref{aa6}, and
equalizing appropriate coefficients on basis elements, we obtain the
system of equations
\begin{eqnarray*}
&&p_{jk,j}p_{ij,i}+p_{jk,k}p_{ik,i}=p_{ij,i}p_{ik,i},\\\label{aa10}
&&p_{ij,i}p_{ik,k}+p_{ij,j}p_{jk,k}=p_{jk,k}p_{ik,k}.
\end{eqnarray*}
The last equalities, due to Lemma \ref{pp0}, are equivalent. This
completes the proof.\end{proof}

%

\begin{remark}
We note that associativity conditions for genetic Volterra algebras
have been found in low dimensional setting in \cite{G2008,MH2015}.
\end{remark}

Recall that if $V$ is a Volterra operator on $S^{m-1}$, then it can
be represented as follows (see \cite{G}):
\begin{equation}\label{VV1}
V(\xb)_k=x_k\bigg(1+\sum_{i=1}^ma_{ik}x_i\bigg), \ \ k=1,\dots,m,
\end{equation}
where $|a_{ik}|\leq 1$ and $a_{ik}=-a_{ki}$. Hence, this
representation allows us for each Volterra operator to assign a
skew-symmetric matrix $(a_{ij})$. This correspondence is a
one-to-one linear transformation. One concludes that the set of
Volterra operators is convex, and its extremal elements are
characterized by $|a_{ik}|=1$ ($k\neq i$). Therefore, a Volterra
algebra $A$, associated with $V$, is called \textit{extremal} if it
corresponds to some extremal Volterra operator.

According to the mentioned correspondence, each Volterra operator
defines a skew-symmetric matrix, which defines some graph \cite{G}.
Namely, we suppose that $a_ {ki}\not =0$
  at  $i\not= k.$ Let us consider a full graph
$G_{m}$ consisting of  $m$  vertices:  $1,2,\dots, m$. Define a
tournament $T_{m},$  as  a  graph  consisting  of  $m$ vertices
labeled  by  $1,2,\dots, m$  corresponding  to  a skew-symmetrical
matrix $(a_{ik})$ according to the following rule: there is an arrow
from  $k$  to  $i$  if  $a_{ki}<0,$ and a reverse arrow otherwise.
Hence, to every Volterra algebra, we associate the constructed
tournament. A number of properties of genetic Volterra algebras can
be obtained by means of the theory of tournaments.

Recall that a tournament is said to be \emph{strong} if it is
possible to go from any vertex to any other vertex according to
directions on the edges. A strong subtournament of the tournament
$T_{m}$ ( $m\geq3$) composed of three vertices is called a
\emph{cyclic triple.}

It is interesting (independent of interest) to get associativity
condition of the Volterra algebra in terms of the corresponding
tournament. We have the following

\begin{corollary}\label{AAA1}
Let $A$ be a Volterra algebra. Then $A$ is associative iff it is an
extremal, and the corresponding tournament doesn't have cyclic
triple.
\end{corollary}

\begin{proof} Assume that $A$ is associative. Then due to Theorem
\ref{aa1} (a) we immediately find the extremity of $A$.  Now we
suppose that the corresponding tournament has a cyclic triple, i.e.
there are three vertices $i,j,k$ which form a cyclic triple. Without
loss of generality (since $i,j,k$ are cyclic tripe), we may choose
$p_{ij,i}=1$, $p_{jk,j}=1$ and $p_{ki,k}=1$. But this contradicts to
Theorem \ref{aa1} (b), i.e.
$p_{jk,j}p_{ij,i}+p_{jk,k}p_{ik,i}=1\neq0=p_{ij,i}p_{ik,i}$.

Now assume that $A$ is extremal, and the corresponding tournament
doesn't have cyclic triple. Let us establish that $A$ is
associative.

 Due to our assumptions, for any three basis elements $\eb_i,
\eb_j, \eb_k$ we have one of the following possibilities:
\begin{enumerate}
\item[{\tt case 1.}] $\eb_{i}\circ \eb_{j}=\eb_{j}$, $\eb_{k}\circ \eb_{j}=\eb_{j}$
and $\eb_{k}\circ \eb_{i}=\eb_{i}$;

\item[{\tt case 2.}] $\eb_{i}\circ \eb_{j}=\eb_{j}$, $\eb_{k}\circ \eb_{j}=\eb_{j}$
and  $\eb_{k}\circ \eb_{i}=\eb_{k}$;

\item[{\tt case 3.}] $\eb_{i}\circ \eb_{j}=\eb_{i}$, $\eb_{k}\circ \eb_{j}=\eb_{j}$
and $\eb_{k}\circ \eb_{i}=\eb_{i}$;

\item[{\tt case 4.}] $\eb_{i}\circ \eb_{j}=\eb_{i}$, $\eb_{k}\circ \eb_{j}=\eb_{k}$
and $\eb_{k}\circ \eb_{i}=\eb_{i}$;

\item[{\tt case 5.}] $\eb_{i}\circ \eb_{j}=e_{j}$, $\eb_{k}\circ \eb_{j}=\eb_{k}$
and $\eb_{k}\circ \eb_{i}=\eb_{k}$;

\item[{\tt case 6.}] $\eb_{i}\circ \eb_{j}=\eb_{i}$, $e_{k}\circ \eb_{j}=\eb_{k}$
and $\eb_{k}\circ \eb_{i}=\eb_{k}$.
\end{enumerate}

For the case 1, we will show the associativity of the vectors
$\eb_i, \eb_j, \eb_k$ (the other cases can be proceeded by the same
argument). Indeed, we have
\begin{eqnarray*}
&&(\eb_j\circ \eb_k)\circ \eb_i=\eb_j\circ \eb_i=\eb_{j}\circ
(\eb_{k}\circ
\eb_{i}),\\
&&(\eb_{k}\circ \eb_i)\circ \eb_{j}=\eb_j=\eb_k\circ
\eb_j=\eb_k\circ(\eb_i\circ \eb_j),
\end{eqnarray*}
these equalities yield the assertion. The proof is complete.
\end{proof}

\begin{definition}
A tournament is called \textit{transitive} if it does not have any
cyclic triple.
\end{definition}

\begin{remark}\label{AAA2}\cite{BM}
There is only one transitive tournament of a given order $m$ (up to
isomorphism).
\end{remark}

Recall that two tournaments are said to be \textit{isomorphic}, if
there exists a permutation of the vertices which transforms one
tournament into the other.

\begin{theorem}\cite{NGS}\label{AAA3} If the tournaments of two extremal Volterra algebras are isomorphic, then the corresponding algebras are isomorphic as well.
\end{theorem}

\begin{theorem}\label{aa11}
Any associative genetic Volterra algebra is isomorphic to the
algebra with structural coefficients
\begin{equation*}\label{aa11}p_{ij,i}= \left\{
\begin{array}{ll}
1 , \ \ \mbox{if} \   i\geq j \\
0 , \ \ \mbox{otherwise} \   \\
\end{array}
\right.
\end{equation*}
\end{theorem}

\begin{proof} It is clear that the corresponding tournament is extremal and transitive, hence due to
Corollary \ref{AAA1} the Volterra algebra is associative. According
to Remark \ref{AAA2} and Theorem \ref{AAA3} we infer the required
assertion.
\end{proof}

\section{Derivations of genetic Volterra algebras}

It is interesting to find all derivations of given algebra. The
well-known Kadison's Theorem states that all derivations of
associative and commutative algebras are trivial. Therefore, under
conditions of Theorem \ref{aa1} any derivation of genetic Volterra
algebra is trivial. In this section, we are going to describe
derivations of genetic Volterra algebras.

Let $A$ be a genetic Volterra algebra generated by the heredity
coefficients $\{p_{ij,k}\}$. We put
$$
I_{i}=\bigg\{j\in I: \ \ p_{ij,i}=\frac{1}{2}\bigg\}.
$$

Any derivation $d$ of $A$ has the following form
\begin{equation}\label{dd1}
D(\eb_{i})=\sum_{j=1}^{m}d_{ij}\eb_{j}, \ \ i\in I, \end{equation}
for some matrix $(d_{ij})$.

\begin{lemma}\label{dd11} Let $A$ be a genetic Volterra algebra, and $d$ be
its derivation given by \eqref{dd1}. If $j\neq i$ and $j\notin
I_{i}$, then $d_{ij}=d_{ji}=0$
\end{lemma}

\begin{proof}
Due to our denotation and the definition of the algebra, one can see
that if $j\notin I_{i}$ then $i\notin I_{j}$. Therefore, it is
enough to establish $d_{ij}=0$.

From the definition of the derivation, we immediately find
\begin{equation}\label{dd12}
D(\eb_{i})=2\eb_{i}\circ D(\eb_{i}).
\end{equation}
From \eqref{hh1} and \eqref{dd1} it follows that
\begin{eqnarray}\label{dd13}
2\eb_{i}\circ D(\eb_{i})&=&2\sum_{j=1}^{m}d_{ij}\eb_{i}\circ
\eb_{j}\nonumber\\[2mm]
&=&2\sum_{j=1}^{m}d_{ij}(p_{ij,i}\eb_{i}+p_{ij,j}\eb_{j})\nonumber\\[2mm]
&=&\bigg(2\sum_{j=1}^{m}d_{ij}p_{ij,i}\bigg)\eb_{i}
+2\sum_{j=1}^{m}d_{ij}p_{ij,j}\eb_{j}
\end{eqnarray}
Now inserting \eqref{dd13} and \eqref{dd1} into \eqref{dd12} and
equalizing the corresponding coefficients on basis elements, we
obtain
\begin{eqnarray*}
d_{ij}\bigg(\frac{1}{2}-p_{ij,j}\bigg)=0, \ \ \ \ j\neq i
\end{eqnarray*}
which implies that $d_{ij}=0$ if $j\notin I_i$. This completes the
proof.
\end{proof}

\begin{corollary}\label{dd2} Let $A$ be a genetic Volterra
algebra. If one has
$$p_{ij,j}\neq \frac{1}{2}$$
for all $i,j\in I$, then any derivations of $A$ is trivial.
\end{corollary}

\begin{proof}
Due to $p_{ij,j}\neq \frac{1}{2}$ from Lemma \ref{dd11} we obtain
$d_{ij}=0$ if $i\neq j$. This means that every derivation has the
following form $D(\eb_{i})=d_{ii}\eb_{i}$. From the definition of
the derivation one has $D(\eb_{i})=2\eb_{i}\circ D(\eb_{i})$. This,
due to $2\eb_{i}\circ D(\eb_{i})=2d_{ii}(\eb_{i}\circ \eb_{i})=2
d_{ii}\eb_{i}$, yields $d_{ii}=0$. The proof is complete.
\end{proof}

Now we are going to describe all nontrivial derivations of three
dimensional genetic Volterra algebras.

\begin{theorem}\label{dd3}  Let $A$ be a three dimensional genetic Volterra algebra.
The algebra has a nontrivial derivation if and only if there exist
$i,j,k$ with $\{i,j,k\}=\{1,2,3\}$ such that $p_{ij,i}=\frac{1}{2}$,
$p_{ik,i}=p_{jk,j}$.
\end{theorem}

\begin{proof} "{\it if}" part. According to Corollary \ref{dd2} we
have that  $I_{j}$ is not empty for some $j\in\{1,2,3\}$. Without
loss of generality we may assume that
$0\neq|I_{1}|\geq|I_2|\geq|I_3|$. Then it is possible that either
$|I_{1}|=1$ or 2. Therefore, we have three main possible cases:
\begin{itemize}
\item[(a)] $|I_{1}|=1$,  $|I_{2}|=1$ and $|I_{3}|=0$;

\item[(b)] $|I_{1}|=2$,  $|I_{2}|=1$ and  $|I_{3}|=1$;

\item[(c)] $|I_{1}|=2$,  $|I_{2}|=2$ and  $|I_{3}|=2$.
\end{itemize}
Note that other cases can be reduced to the listed ones.

In the considered setting, the derivation has a form
\begin{equation}\label{dd31}
D(\eb_{i})=\sum_{j=1}^{3}d_{ij}\eb_{j}.
\end{equation}

Now let us investigate above listed cases one by one.

 {\bf Case
(a)}. Assume that $|I_{1}|=1$, $|I_{2}|=1$ and $|I_{3}|=0$. Then we
have
$$p_{12,1}=p_{12,2}=\frac{1}{2},\ p_{13,1}=1-p_{13,3}\neq\frac{1}{2}$$ and
$p_{23,2}=1-p_{23,3}\neq\frac{1}{2}$. Hence, due to Lemma \ref{dd1}
one finds $d_{13}=d_{31}=d_{23}=d_{32}=0$. Using the same argument
as in the proof of Corollary \ref{dd2} one can show that $d_{33}=0$.
Hence, we obtain $D(\eb_{3})=0$.

In the considered case, we have
$D(\eb_{1})=d_{11}\eb_{1}+d_{12}\eb_{2}$. Keeping in mind the
equality
\begin{eqnarray*}
2\eb_{1}\circ D(\eb_{1})&=&2(d_{11}\eb_{1}+d_{12}\eb_{2})\circ
\eb_{1}\\
&=&2d_{11}\eb_{1}+d_{12}\eb_{1}+d_{12}\eb_{2}
\end{eqnarray*}
from  $D(\eb_{1})=2\eb_{1}\circ D(\eb_{1})$ one finds
$d_{11}=-d_{12}$. Let us denote $a:=d_{11}$.

Using the same argument with $\eb_{2}$, we get $d_{21}=-d_{22}$.
Hence, we denote $b:=d_{21}$.

It is clear that $D(\eb_{2}\circ \eb_{3})=\eb_{3}\circ
D(\eb_{2})+D(\eb_{3})\circ \eb_{2}$. Therefore, due to
\begin{eqnarray*}
D(\eb_{2}\circ
\eb_{3})&=&D(p_{23,2}\eb_{2}+p_{23,3}\eb_{3})\\[2mm]
&=&p_{23,2}D(\eb_{2})+p_{23,3}D(\eb_{3})\\
&=&p_{23,2}D(\eb_{2})\\
&=&bp_{23,2}\eb_{1}-bp_{23,2}\eb_{2}
\end{eqnarray*}
and
\begin{eqnarray*}
\eb_{3}\circ D(\eb_{2})+D(\eb_{3})\circ \eb_{2}&=&\eb_{3}\circ
D(\eb_{2})\\
&=&b(\eb_{3}\circ \eb_{1}-\eb_{3}\circ
\eb_{2})\\
&=&b(p_{13,1}\eb_{1}+p_{13,3}\eb_{3}-p_{23,2}\eb_{2}-p_{23,3}\eb_{3})
\end{eqnarray*}
we obtain $b(p_{13,3}-p_{23,3})=0$.

Now using the same argument with $D(\eb_{1}\circ
\eb_{3})=\eb_{3}\circ D(\eb_{1})+D(\eb_{3})\circ \eb_{1}$ one finds
$a(p_{13,3}-p_{23,3})=0$. Since at least one of $a$ and $b$ is non
zero, hence we have $p_{13,3}=p_{23,3}$. This means that
$1-p_{13,3}=1-p_{23,3}$, so
$p_{13,1}=p_{23,2}$.\\

{\bf Case (b).} Let us suppose that $|I_{1}|=2$, $|I_{2}|=1$ and
$|I_{3}|=1$. Then we have $p_{12,1}=p_{12,2}=\frac{1}{2}$,
$p_{13,1}=p_{13,3}=\frac{1}{2}$ and
$p_{23,2}=1-p_{23,3}\neq\frac{1}{2}$. So, from Lemma \ref{dd1} one
concludes that $d_{23}=d_{32}=0$.

From $D(\eb_{3})=d_{31}\eb_{1}+d_{33}\eb_{3}$ and
\begin{eqnarray*}
2\eb_{3}\circ D(\eb_{3})&=&2(d_{31}\eb_{1}+d_{33}\eb_{3})\circ
\eb_{3}\\
&=&2d_{31}\eb_{1}+d_{33}\eb_{1}+d_{3}\eb_{3}
\end{eqnarray*}
with $D(\eb_{3})=2\eb_{3}\circ D(\eb_{3})$ we obtain
$d_{31}=-d_{33}$. Hence, we denote $a:=d_{31}$.

Using the same argument with $\eb_{2}$, one gets $d_{21}=-d_{22}$,
so one denotes $b:=d_{21}$.

Inserting the equalities
$D(\eb_{1})=d_{11}\eb_{1}+d_{12}\eb_{2}+d_{13}\eb_{3}$ and
\begin{eqnarray*}
2\eb_{1}\circ
D(\eb_{1})&=&2(d_{11}\eb_{1}+d_{12}\eb_{2}+d_{13}\eb_{3})\circ
\eb_{1}\\
&=&(2d_{11}+d_{12}+d_{13})\eb_{1}+d_{12}\eb_{1}+d_{12}\eb_{2}+d_{13}\eb_{3}
\end{eqnarray*}
into $D(\eb_{1})=2\eb_{1}\circ D(\eb_{1})$, and equalizing
corresponding coefficients at $\eb_{1}$ we find
$d_{11}+d_{12}=-d_{13}$. So, we denote $c:=d_{11}$ and $d:=d_{12}$.

It is clear that
\begin{eqnarray}\label{dd32}D(\eb_{2}\circ \eb_{3})=\eb_{3}\circ
D(\eb_{2})+D(\eb_{3})\circ \eb_{2}.
\end{eqnarray}

On the other hand, we have
\begin{eqnarray*}
D(\eb_{2}\circ
\eb_{3})&=&D(p_{23,2}\eb_{2}+p_{23,3}\eb_{3})\\
&=&p_{23,2}D(\eb_{2})+p_{23,3}D(\eb_{3})\\
&=&p_{23,2}(b\eb_{1}-b\eb_{2})+p_{23,3}(a\eb_{1}-a\eb_{3}).
\end{eqnarray*}
Similarly, one finds
\begin{eqnarray*}
\eb_{3}\circ D(\eb_{2})+D(\eb_{3})\circ
\eb_{2}&=&\eb_{3}\circ(b\eb_{1}-b\eb_{2})+\eb_{2}\circ(a\eb_{1}-a\eb_{3})\\
&=&b(\eb_{3}\circ \eb_{1}-\eb_{3}\circ \eb_{2})+a(\eb_{2}\circ
\eb_{1}-\eb_{2}\circ \eb_{3})\\
&=&b\bigg(\frac{\eb_{1}+\eb_{3}}{2}-p_{23,2}\eb_{2}-p_{23,3}\eb_{3}\bigg)\\
&&+a\bigg(\frac{\eb_{1}+\eb_{2}}{2}-p_{23,2}\eb_{2}-p_{23,3}\eb_{3}\bigg)
\end{eqnarray*}

Now inserting the last equalities into \eqref{dd32}, we obtain
$$-bp_{23,2}=a\bigg(\frac{1}{2}-p_{23,2}\bigg)-bp_{23,2},\ \  \Longrightarrow  a\bigg(\frac{1}{2}-p_{23,2}\bigg)=0,$$
$$-ap_{23,3}=b\bigg(\frac{1}{2}-p_{23,2}\bigg)-ap_{23,3},\ \  \Longrightarrow b\bigg(\frac{1}{2}-p_{23,3}\bigg)=0.$$

From $3\notin I_{2}$ one gets $a=b=0$. Hence, we have
\begin{eqnarray*}
&&D(\eb_{1})=c\eb_{1}+d\eb_{2}-(c+d)\eb_{3},\\ &&D(\eb_{2})=0,\\
&&D(\eb_{3})=0.
\end{eqnarray*}

We know that $D(\eb_{1}\circ \eb_{3})=\eb_{3}\circ
D(\eb_{1})+D(\eb_{3})\circ \eb_{1}$. Hence from
\begin{eqnarray*}
D(\eb_{1}\circ
\eb_{3})&=&D\bigg(\frac{1}{2}\eb_{1}+\frac{1}{2}\eb_{3}\bigg)\\
&=&\frac{1}{2}D(\eb_{1})\\
&=&\frac{c}{2}\eb_{1}+\frac{d}{2}\eb_{2}-\frac{c+d}{2}\eb_{3}
\end{eqnarray*}
and
\begin{eqnarray*}
\eb_{3}\circ D(\eb_{1})+D(\eb_{3})\circ \eb_{1}&=&\eb_{3}\circ
D(\eb_{1})\\
&=&c(\eb_{3}\circ \eb_{1})+d(\eb_{3}\circ
\eb_{2})-(c+d)(\eb_{3}\circ
\eb_{3})\\
&=&\frac{c}{2}\eb_{1}+\frac{c}{2}\eb_{3}+dp_{23,2}\eb_{2}+dp_{23,3}\eb_{3}-(c+d)\eb_{3}
\end{eqnarray*}
one finds $d(1/2-p_{23,2})=0$. From $3\notin I_{2}$ it follows that
$d=0$. This implies $D(\eb_{1})=c\eb_{1}-c\eb_{3}$, $D(\eb_{2})=0$
and $D(\eb_{3})=0$.

Similarly, from $D(\eb_{1}\circ \eb_{2})=\eb_{2}\circ
D(\eb_{1})+D(\eb_{2})\circ \eb_{1}$ with
\begin{eqnarray*}
D(\eb_{1}\circ
\eb_{2})&=&D\bigg(\frac{1}{2}\eb_{1}+\frac{1}{2}\eb_{2}\bigg)\\
&=&\frac{1}{2}D(\eb_{1})\\
&=&\frac{c}{2}\eb_{1}-\frac{c}{2}\eb_{3},
\end{eqnarray*}
and
\begin{eqnarray*}
\eb_{2}\circ D(\eb_{1})+D(\eb_{2})\circ \eb_{1}&=&\eb_{2}\circ
D(\eb_{1})\\
&=&c(\eb_{2}\circ \eb_{1})-c(\eb_{2}\circ \eb_{3})\\
&=&\frac{c}{2}\eb_{1}+\frac{c}{2}\eb_{2}-cp_{23,2}\eb_{2}-cp_{23,3}\eb_{3}
\end{eqnarray*}
we obtain  $c(1/2-p_{23,2})=0$. Taking into account $3\notin I_{2}$,
one finds $c=0$. Hence, $D(\eb_{1})=0$, $D(\eb_{2})=0$ and
$D(\eb_{3})=0$. This means that the algebra $A$ has only trivial
derivation, which contradicts to our assumption.

{\bf Case (c)}. In this case, we have $|I_{1}|=2$, $|I_{2}|=2$ and
$|I_{3}|=2$, which yield
$p_{12,1}=p_{12,2}=p_{13,1}=p_{13,3}=p_{23,2}=p_{23,3}=\frac{1}{2}$.
This implies the required assertion.\\

"{\it only if}" part. Assume that $p_{12,1}=\frac{1}{2}$,
$p_{13,1}=p_{23,2}$. We are going to show that the existence of non
trivial derivation of $A$. There are two possibilities:

\begin{enumerate}

\item[(A)] $p_{12,1}=\frac{1}{2}$,
$p_{13,1}=p_{23,2}\neq\frac{1}{2}$;

\item[(B)] $p_{12,1}=p_{13,1}=p_{23,2}=\frac{1}{2}$.
\end{enumerate}

In case (A) we define a linear mapping $D$ as follows:
\begin{equation}\label{dd33}
D(\eb_{1})=a(\eb_{1}-\eb_{2}), \ \ D(\eb_{2})=b(\eb_{1}-\eb_{2}), \
\ D(\eb_{3})=0,
\end{equation}
where $a,b\in\br$. One can check that the defined mapping is a
derivation.

In case (B) we define a linear mapping $D$ as follows:
\begin{equation}\label{dd34}
D(\eb_{1})=c\eb_{1}+d\eb_{2}-(c+d)\eb_{3}, \ \
D(\eb_{2})=b(\eb_{1}-\eb_{2}), \ \ D(\eb_{3})=a(\eb_{1}-\eb_{3}),
\end{equation}
where $a,b,c,d\in\br$. One can check that the defined mapping is a
derivation. This completes the proof.
\end{proof}

\begin{lemma}\label{dd3c}  Let $A$ be a three dimensional genetic Volterra algebra such that
$p_{12,1}=\frac{1}{2}.$ Then the following statements hold:
\begin{itemize}
\item[(A)] If $p_{13,1}=p_{23,2}\neq\frac{1}{2},$ then each derivation
of $A$ is of the form
$$D(\eb_{1})=a(\eb_{1}-\eb_{2}),  \  \ D(\eb_{2})=b(\eb_{1}-\eb_{2}),\ \  D(\eb_{3})=0.$$

\item[(B)] If $p_{13,1}=p_{23,2}=\frac{1}{2},$ then each derivation $D$
of $A$ is defined by
\begin{equation}\label{der}
D(\eb_i)=\sum_{j=1}^{3}d_{ij}\eb_{j} \end{equation} where
$\sum\limits_{j=1}^{3}d_{ij}=0$ for $i=1,2,3.$
\end{itemize}
\end{lemma}

\begin{proof}
(A). Let $D(\eb_i)=\sum_{j=1}^{3}d_{ij}\eb_{j}$ ($i=1,2,3$) be a
derivation. From the condition (A) we infer that
$d_{13}=d_{31}=d_{23}=d_{32}=d_{22}=0$. The equality
$$d_{33}\eb_{3}=D(\eb_{3})=D(\eb_{3}\circ \eb_{3})=2\eb_{3}\circ D(\eb_{3})=2\eb_{3}\circ d_{33}\eb_{3},$$
implies $d_{33}=0$. Using the relations $D(\eb_1)=2\eb_1\circ
D(\eb_1)$ and $D(\eb_2)=2\eb_2\circ D(\eb_2)$ we get that
$d_{11}=-d_{12}=:a$ and $d_{21}=-d_{22}=:b$. One can check that the
defined mapping is indeed a derivation.

(B) Let $D(\eb_i)=\sum_{j=1}^{3}d_{ij}\eb_{j}$($i=1,2,3$) be a
derivation. Then using the given condition, we have
\begin{eqnarray*}
\sum_{j=1}^{3}d_{ij}\eb_{j}&=&d(\eb_{i})=D(\eb_{i}\circ
\eb_{i})=2\eb_{i}\circ
D(\eb_{i})\\
&=&2\eb_{i}\circ\bigg(\sum_{j=1}^{3}d_{ij}\eb_{j}\bigg)=\sum_{j=1}^{3}d_{ij}\eb_{i}+\sum_{j=1}^{3}d_{ij}\eb_{j}.
\end{eqnarray*}
which yields  $\sum_{j=1}^{3}d_{ij}=0$ for $i=1,2,3.$

Now let us establish that the reverse. Namely, we show that the
mapping given by \eqref{der} with $\sum_{j=1}^{3}d_{ij}=0$ for
$i=1,2,3$ is a derivation. Indeed, we have
\begin{eqnarray*}
D(\eb_{i}\circ
\eb_{k})&=&D\bigg(\frac{1}{2}(\eb_{i}+\eb_{k})\bigg)=\frac{1}{2}\bigg(\sum_{j=1}^{3}d_{ij}\eb_{j}+\sum_{j=1}^{3}d_{kj}\eb_{j}\bigg)\\
&=&\frac{1}{2}\bigg(\sum_{j=1}^{3}d_{ij}\eb_{j}+\sum_{j=1}^{3}d_{kj}\eb_{j}\bigg)+\frac{1}{2}\bigg(\sum_{j=1}^{3}d_{ij}\eb_{i}+\sum_{j=1}^{3}d_{kj}\eb_{k}\bigg)\\[2mm]
&=&\eb_{i}\circ D(\eb_{k})+\eb_{k}\circ D(\eb_{i}).
\end{eqnarray*}
This completes the proof.

\end{proof}
Now we want to apply Lemma  \ref{dd3c} to describe all local
derivations of the three dimensional genetic Volterra algebra.
Recall that a linear map $\Delta:A\to A$ is called \textit{local
derivation} if for any $x\in A$ there exists a derivation $D_x:A\to
A$ such that $\Delta(x)=D_x(x)$.

\begin{theorem}
Let $A$ be a three dimensional genetic Volterra algebra. Then any
local derivation is a derivation.
\end{theorem}

\begin{proof} Due to Lemma \ref{dd3c} a derivation on $A$ exists if
one of the following cases hold:
\begin{enumerate}
\item[(A)]$p_{12,1}=\frac{1}{2}$, $p_{13,1}=p_{23,2}\neq\frac{1}{2}$;

\item[(B)] $p_{12,1}=p_{13,1}=p_{23,2}=\frac{1}{2}$.
\end{enumerate}

Let us consider the mentioned cases one by one.

Assume that case (A) holds. Then due to Lemma \ref{dd3c} any
derivation of $A$ has a form given by \eqref{dd33}.

Suppose that $\Delta$ is a local derivation of $A$. Then from the
definitions of local derivation, one finds
  \begin{equation}\label{dd333}
\Delta(\eb_{1})= D_{\eb_{1}}(\eb_{1}), \ \ \Delta(\eb_{2})=
D_{\eb_{2}}(\eb_{2}),  \ \ \Delta(\eb_{3})=D_{\eb_{3}}(\eb_{3}).
\end{equation}

By means of \eqref{dd33} one infers that
\begin{equation}\label{dd3331}
D_{\eb_{1}}(\eb_{1})=a_1(\eb_{1}-\eb_{2}), \ \
D_{\eb_{2}}(\eb_{2})=b_2(\eb_{1}-\eb_{2}),  \ \
D_{\eb_{3}}(\eb_{3})=0.
\end{equation}
Therefore, from \eqref{dd333} and \eqref{dd3331} we find
 \begin{equation*}\label{dd3332}
\Delta(\eb_{1})= a_1(\eb_{1}-\eb_{2}), \ \ \Delta(\eb_{2})=
b_2(\eb_{1}-\eb_{2}),  \ \ \Delta(\eb_{3})=0
\end{equation*}
which according to Lemma \ref{dd3c} means that $\Delta$ is a
derivation.

Now assume that case (B) holds. Then for a local derivation $\Delta$
we have \eqref{dd333}. From \eqref{dd34} one finds
  \begin{equation*}\label{dd3333}
\Delta(\eb_{1})= c_1\eb_{1}+d_1\eb_{2}-(c_1+d_1)\eb_{3}, \ \
\Delta(\eb_{2})= b_2(\eb_{1}-\eb_{2}),  \ \
\Delta(\eb_{3})=a_3(\eb_{1}-\eb_{3}).
\end{equation*}
This again by Lemma \ref{dd3c} yields that $\Delta$ is a
derivation.\end{proof}

From this theorem we can formulate the following

\begin{conjecture}
Let $A$ be an $n$-dimensional genetic Volterra algebra. Then a
category of local derivations of $A$ coincides with the category of
derivations of $A$.
\end{conjecture}

\section*{Acknowledgments}
The authors are grateful to an anonymous referee whose valuable
comments and remarks improved the presentation of this paper.

\end{document}